\documentclass[a4paper,12pt,reqno]{amsart}
\usepackage{lmodern,amsmath,amssymb,mathtools,amsthm}
\usepackage[utf8]{inputenc}
\usepackage{enumerate}
\usepackage{graphics,graphicx}
\usepackage{bigints}  
\usepackage{pdfpages}
\usepackage{morefloats}
\usepackage{geometry} 
\usepackage{wrapfig}
\usepackage{wasysym}
\usepackage{cleveref}
\usepackage{floatrow,morefloats}
\usepackage{epsfig}
\usepackage{epstopdf}
\usepackage{caption,subcaption}
\usepackage{array}
\newcolumntype{M}[1]{>{\centering\arraybackslash}m{#1}}

\newtheorem{lemma}{Lemma}
\newtheorem{corollary}{Corollary}
\newtheorem{definition}{Definition}
\newtheorem{theorem}{Theorem}
\newtheorem{remark}{Remark}

\numberwithin{example}{section}
\numberwithin{equation}{section}
\numberwithin{theorem}{section}
\numberwithin{lemma}{section}
\numberwithin{definition}{section}
\numberwithin{corollary}{section}
\newtheorem*{problem}{Problem}

\usepackage{blindtext}
\title{Radius Problems For Functions Associated with a Nephroid Domain}

\author{Lateef Ahmad Wani$^\dagger$}
\address{$^\dagger$Department of Mathematics\\ Indian Institute of Technology, Roorkee-247667, Uttarakhand, India}
\email{lateef17304@gmail.com}

\author{A. Swaminathan$^\ddagger$}
\address{$^\ddagger$Department of Mathematics\\ Indian Institute of Technology, Roorkee-247667, Uttarakhand, India}
\email{mathswami@gmail.com, a.swaminathan@ma.iitr.ac.in}

\allowdisplaybreaks
\bigskip

\begin{document}
	
\pagestyle{myheadings}
\begin{abstract}
Let $\mathcal{S}^*_{Ne}$ be the collection of all analytic functions $f(z)$ defined on the open unit disk $\mathbb{D}$ and satisfying the normalizations $f(0)=f'(0)-1=0$ such that the quantity $zf'(z)/f(z)$ assumes values from the range of the function $\varphi_{\scriptscriptstyle{Ne}}(z):=1+z-z^3/3\,,z\in\mathbb{D}$, which is the interior of the nephroid given by
\begin{align*}
\left((u-1)^2+v^2-\frac{4}{9}\right)^3-\frac{4 v^2}{3}=0.
\end{align*}
In this work, we find sharp $\mathcal{S}^*_{Ne}$-radii for several geometrically defined function classes introduced in the recent past. In particular, $\mathcal{S}^*_{Ne}$-radius for the starlike class $\mathcal{S}^*$ is found to be $1/4$. Moreover, radii problems related to the families defined in terms of ratio of functions are also discussed. Sharpness of certain radii estimates are illustrated graphically.
\end{abstract}

\maketitle

\pagestyle{myheadings}

\markboth{Lateef Ahmad Wani and A. Swaminathan}{Radius Problems Associated with a Nephroid Domain}

\section{Introduction} \label{Section-Introduction-RRN}
Let $\mathcal{A}$ be the class of all analytic functions satisfying the conditions $f(0)=0$ and $f'(0)=1$ in the open unit disc $\mathbb{D}:=\left\{z:|z|<1\right\}$. Clearly, for each $f\in\mathcal{A}$, the function $\mathcal{Q}_f(z):\mathbb{D}\to\mathbb{C}$ given by
\begin{align} \label{Definition-Basic-Fraction}
\mathcal{Q}_f(z):=\frac{zf'(z)}{f(z)}
\end{align}
is analytic and satisfies $\mathcal{Q}_f(0)=1$.
Let $\mathcal{S}\subset\mathcal{A}$ be the family of univalent functions, and $\mathcal{S}^*(\alpha)\subset\mathcal{S}$ be the family of starlike functions of order $\alpha\,(0\leq\alpha<1)$ given by the analytic characterization
\begin{align*}
\mathcal{S}^*(\alpha):=\left\{f\in\mathcal{A}:\mathrm{Re}\left(\mathcal{Q}_f(z)\right)>\alpha\right\}.
\end{align*}
Further, let us define the class $\mathcal{C}(\alpha)$ by the relation: $f\in\mathcal{C}(\alpha) \iff zf'\in\mathcal{S}^*(\alpha)$. The functions in $\mathcal{S}^*:=\mathcal{S}^*(0)$ and $\mathcal{C}:=\mathcal{C}(0)$ are, respectively, starlike and convex in $\mathbb{D}$. For functions $f$ and $g$ analytic on $\mathbb{D}$, we say that $f$ is {\it subordinate} to $g$, written $f\prec{g}$, if there exists an analytic function $w$ satisfying $w(0)=0$ and $|w(z)|<1$ such that $f(z)=g(w(z))$. Indeed, $f\prec g \implies f(0)=g(0)$ and $f(\mathbb{D})\subset g(\mathbb{D})$. Furthermore, if the function $g$ is univalent, then $f\prec g \iff f(0)=g(0)$ and $f(\mathbb{D})\subset g(\mathbb{D})$.
By unifying several earlier results on subordination, Ma and Minda \cite{Ma-Minda-1992-A-unified-treatment} introduced the function class $\mathcal{S}^*(\varphi)$ which, for brevity, we write as definition.
\begin{definition}\label{Definition-Ma-Minda-classes-Full}
Let $\mathcal{S}^*(\varphi)$ denote the class of functions characterized as	
\begin{align}\label{Definition-Ma-Minda-classes-varphi}
\mathcal{S}^*(\varphi):=\left\{f\in\mathcal{A}:\mathcal{Q}_f(z)\prec\varphi(z)\right\},
\end{align}
where the analytic function $\varphi:\mathbb{D}\to\mathbb{C}$ is required to satisfy the following conditions:
\begin{enumerate}[{\rm(i)}]
	\item
	$\varphi(z)$ is univalent with $\mathrm{Re}(\varphi)>0$,
	\item
	$\varphi(\mathbb{D})$ is starlike with respect to $\varphi(0)=1$,
	\item
	$\varphi(\mathbb{D})$ is symmetric about the real line, and
	\item
	$\varphi'(0)>0$.
\end{enumerate}
\end{definition}
Evidently, for every such $\varphi$, $\mathcal{S}^*(\varphi)$ is always a subclass of the class of starlike functions $\mathcal{S}^*$, and $\mathcal{S}^*(\varphi)=\mathcal{S}^*$ for $\varphi(z)=(1+z)/(1-z)$. We call $\mathcal{S}^*(\varphi)$ as the Ma-Minda type starlike class associated with $\varphi$. Specializing the function $\varphi$ in \eqref{Definition-Ma-Minda-classes-varphi} yields several important function classes. For instance,  $\mathcal{S}^*(\alpha):=\mathcal{S}^*\left(\frac{1+(1-2\alpha)z}{1-z}\right)$, the Janowski class $\mathcal{S}^*[A,B]:=\mathcal{S}^*\left(\frac{1+Az}{1+Bz}\right)$, where $-1\leq{B}<A\leq1$, and many more.
This interesting approach of constructing starlike classes has attracted many researchers, as a result, several Ma-Minda type starlike classes have been introduced and studied recently by modifying the right side function $\varphi$ in \eqref{Definition-Ma-Minda-classes-varphi} appropriately. Below we mention a few of the recently introduced Ma-Minda type classes which will be later used in our discussion.
\begin{enumerate}[(1)]
\item
	$\mathcal{S}^*_L:=\mathcal{S}^*(\sqrt{1+z})$. This class was first introduced by Sok\'{o}{\l} and Stankiewicz \cite{Sokol-J.Stankwz-1996-Lem-of-Ber}, and later discussed extensively by many authors, see \cite{Ali-Jain-Ravi-2012-Radii-LemB-AMC,Cang-Liu-2015-Radius-Convexity-LemB-Expo-Math, Sokol-2009-Radius-Problems-SL*-AMC, Sokol-Thomas-2018-LemB-Houston-Journal} and references therein. The function $\varphi_{\scriptscriptstyle{L}}(z):=\sqrt{1+z}$ maps $\mathbb{D}$ onto the region bounded by the right-loop of {\it lemniscate of Bernoulli} $|w^2-1|=1$.
\item
	$\mathcal{S}^*_{RL}:=\mathcal{S}^*(\varphi_{\scriptscriptstyle{RL}})$ with
	\begin{align} \label{Definition-Shifted-Lem-Ber}
	\varphi_{\scriptscriptstyle{RL}}(z):=\sqrt{2}-(\sqrt{2}-1)\sqrt{\frac{1-z}{1+2(\sqrt{2}-1)z}}, \qquad z\in\mathbb{D},
	\end{align}
	was introduced by Mendiratta et al. \cite{Mendiratta-2014-Shifted-Lemn-Bernoulli}. The function $\varphi_{\scriptscriptstyle{RL}}(z)$ maps $\mathbb{D}$ univalently onto the interior of left-loop of shifted lemniscate of Bernoulli $|(w-\sqrt{2})^2-1|=1$.
\item
	$\mathcal{S}^*_{\leftmoon}:=\mathcal{S}^*(z+\sqrt{1+z^2})$. Introduced by Raina and Sok\'{o}{\l} \cite{Raina-Sokol-2015-Crescent-Shaped-I}, the analytic function $\varphi_{\scriptscriptstyle{\leftmoon}}(z):=z+\sqrt{1+z^2}$
	sends $\mathbb{D}$ onto a {\it crescent} shaped region.
\item
	$\mathcal{S}^*_{e}:=\mathcal{S}^*(e^{z})$ was considered by Mendiratta et al. \cite{Mendiratta-Ravi-2015-Expo-BMMS}.
\item
	$\mathcal{S}^*_C:=\mathcal{S}^*(1+4z/3+2z^2/3)$ was introduced by Sharma et al. \cite{Sharma-Ravi-2016-Cardioid}. The function $\varphi_{\scriptscriptstyle{C}}(z):=1+4z/3+2z^2/3$ sends $\mathbb{D}$ onto a region bounded by a {\it cardioid}.
\item
	$\mathcal{S}^*_{R}:=\mathcal{S}^*(\varphi_{\scriptscriptstyle{0}})$, where $\varphi_{\scriptscriptstyle{0}}(z)$ is the rational function
	\begin{align*}
	\varphi_{\scriptscriptstyle{0}}(z):=
	1+\frac{z}{k}\left(\frac{k+z}{k-z}\right)=1+\frac{1}{k}z+\frac{2}{k^2}z^2+\cdots,
	\quad k=\sqrt{2}+1.
	\end{align*}
	The class $\mathcal{S}^*_{R}$ was introduced by Kumar and Ravichandran \cite{Kumar-Ravi-2016-Starlike-Associated-Rational-Function}.
\item
	$\mathcal{BS}^*(\alpha):=\mathcal{S}^*\left(1+{z}/{(1-\alpha{z}^2)}\right)$, $0\leq\alpha<1$, is the class associated with the {\it Booth lemniscate} which was introduced by Kargar et al. \cite{Kargar-2019-Booth-Lem-A.M.Physics}.
\item
	$\mathcal{S}^*_S:=\mathcal{S}^*(1+\sin{z})$ was introduced by Cho et al. \cite{Cho-2019-Sine-BIMS}. The boundary of $\varphi_{\scriptscriptstyle{S}}(\mathbb{D})$, where $\varphi_{\scriptscriptstyle{S}}(z):=1+\sin{z}$, is an {\it eight-shaped} curve.
\item
	For $0\leq\alpha<1$, Khatter et al. \cite{Khatter-Ravi-2019-Lem-Exp-Alpha-RACSAM} introduced
	\begin{align*}
	\mathcal{S}^*_{\alpha,e}:=\mathcal{S}^*\left(\alpha+(1-\alpha)e^z\right)
	\text{ and }
	\mathcal{S}^*_L(\alpha):=\mathcal{S}^*\left(\alpha+(1-\alpha)\sqrt{1+z}\right).
	\end{align*}
	For $\alpha=0$, the above classes reduce to $\mathcal{S}^*_e$ and $\mathcal{S}^*_L$, respectively.
\item
	Very recently, Wani and Swaminathan \cite{Wani-Swami-Nephroid-Basic} introduced the function class
	\begin{align*}
	\mathcal{S}^*_{Ne}:=\mathcal{S}^*\left(1+z-z^3/3\right).
	\end{align*}
	The authors proved that the analytic function $\varphi_{\scriptscriptstyle{Ne}}(z):=1+z-z^3/3$
	maps $\mathbb{D}$ univalently onto the interior of a 2-cusped {\it kidney-shaped} curve called {\it nephroid} given by
	\begin{align} \label{Equation-of-Nephroid}
	\left((u-1)^2+v^2-\frac{4}{9}\right)^3-\frac{4 v^2}{3}=0.
	\end{align}
\end{enumerate}
Apart from studying several characteristic properties of the region bounded by the nephroid curve \eqref{Equation-of-Nephroid} (see \Cref{Figure-Nephroid}), the authors in \cite{Wani-Swami-Nephroid-Basic} discussed certain inclusion relations, coefficient problems, and a few subordination results related to the function class $\mathcal{S}^*_{Ne}$.

\begin{figure}[H]
	\includegraphics{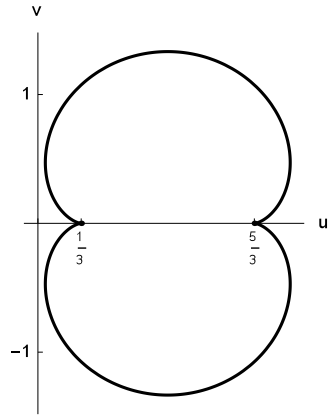}
	\caption{Nephroid: $\left((u-1)^2+v^2-4/9\right)^3-4 v^2/3=0$.}
	\label{Figure-Nephroid}
\end{figure}

\section{Radius Problems}
Let $\mathcal{F}$ and $\mathcal{G}$ be two subclasses of $\mathcal{A}$. Then the $\mathcal{F}$-radius for the class $\mathcal{G}$ is the largest number $\rho\in(0,1)$ such that $r^{-1}f(rz)\in\mathcal{F}$ for all $f\in\mathcal{G}$, where $0<r\leq\rho$. Moreover, the number $\rho$ is said to be {\it sharp} if there exists a function $f_0\in\mathcal{G}$ such that $r^{-1}f_0(rz)\not\in\mathcal{F}$ whenever $r>\rho$ and such a function $f_0$ is called {\it extremal function}. The problem of finding the number ``$\rho$" is called a radius problem in geometric function theory (GFT). In the sequel, by
$${R}_{\mathcal{F}}(\mathcal{G})=\rho,$$
we mean that $\rho$ is the $\mathcal{F}$-radius of $\mathcal{G}$. Below we present a table depicting several radii results related to some of the most important families in univalent function theory. For further details, we refer to \cite[Chapter 13]{Goodman-Book-UFs-II-1983}.

\begin{center}
	\begin{table}[H]
		\centering
	\begin{tabular}{ |c|c|M{1.5cm}|m{4.5cm}| }
		\hline
	  {\boldmath{${R}_{\mathcal{F}}(\mathcal{G})$}}            & {\boldmath{$\rho$}}  & {\bf Approx. value}  & {\bf Extremal function} \\
		\hline
		${R}_{\mathcal{C}}(\mathcal{S})={R}_{\mathcal{C}}(\mathcal{S}^*)$   & $2-\sqrt{3}$         & 0.267          & $\frac{z}{(1-z)^2}$ \\
	    \hline
	    ${R}_{\mathcal{C}}\left(\mathcal{S}^*(1/2)\right)$ & $(2\sqrt{3}-3)^{1/2}$     & 0.68125         &    $\frac{z}{1-z}$          \\
		\hline
		${R}_{\mathcal{C}(\alpha)}(\mathcal{S}^*)$ & $\frac{2-\sqrt{3+\alpha^2}}{1+\alpha}$ & ~ & $\frac{z}{(1-z)^2}$  \\
		\hline
		${R}_{\mathcal{C}(\alpha)}(\mathcal{C})$   & $\frac{1-\alpha}{1+\alpha}$            & ~ &
			$\frac{1-(1-z)^{2\alpha-1}}{2\alpha-1}$, if $\alpha\neq1/2$
			$-\ln(1-z)$, if $\alpha=1/2$
			\\
		\hline
	\end{tabular}
\caption{Radii results of well-known classes}
\end{table}
\label{table:I}
\end{center}

Solving the radius problems for classes of functions with nice geometries is continuing to be an active area of research in GFT. Investigating the radii problems became a bit easier after the dawn of subordination theory, particularly in the cases where the class $\mathcal{F}$ has a Ma-Minda type representation \eqref{Definition-Ma-Minda-classes-varphi}. For such classes, we redefine the radius problem in geometrical terms as:
\begin{definition}
	Let $\mathbb{D}_r:=\{z:|z|<r\}$ and $\Omega_{\varphi}:=\varphi(\mathbb{D})$. By $\mathcal{S}^*(\varphi)$-radius for the class $\mathcal{G}\subset\mathcal{A}$, denoted by ${R}_{\mathcal{S}^*_{\varphi}}(\mathcal{G})$, we mean the largest number $\rho\in(0,1)$ such that each $f\in\mathcal{G}$ satisfies
	\begin{align*}
	\mathcal{Q}_f(\mathbb{D}_r)\subset\Omega_{\varphi} ~ ~ \text{\ for every }~ r\leq\rho,
	\end{align*}
	where $\mathcal{Q}_f(z)$ is defined in \eqref{Definition-Basic-Fraction}. Moreover, the number $\rho$ is said to be sharp if we can find a function $f_0\in\mathcal{G}$ such that
	\begin{align*}
	\mathcal{Q}_{f_0}(\partial\mathbb{D}_\rho)\cap\partial\Omega_{\varphi}\neq\emptyset.
	\end{align*}
\end{definition}
For certain recent results in this direction, we provide the following table along with the respective references. Also, see
\cite{Ali-Jain-Ravi-2012-Radii-LemB-AMC, Ali-Jain-Ravi-2013-Radius-Constants-BMMS,  Khatter-Ravi-2019-Lem-Exp-Alpha-RACSAM,Ravi-Ronning-1997-Complex-Variables}
and the references therein.
\begin{center}
	\begin{table}[H]
		\centering
		\begin{tabular}{ |c|c|M{1.5cm}|M{2.8cm}|c| }
			\hline
			{\boldmath{${R}_{\mathcal{S}^*_{\varphi}}(\mathcal{G})$}}	& {\boldmath{$\rho$}} & {\bf Approx. value} & {\bf Extremal function}& {\bf Refer} \\
			\hline
			${R}_{\mathcal{S}^*_{\leftmoon}}(\mathcal{S}^*)$ & $\sqrt{2}-1$ & $0.41421$ & $\frac{z}{(1-z)^2}$ & \cite{Gandhi-Ravi-2017-Lune}    \\
			\hline
			${R}_{\mathcal{S}^*_{RL}}(\mathcal{S}^*)$ & {\scriptsize{$\sqrt{(\sqrt{5}-2)(5\sqrt{2}-7)}$}} & $0.129525$ & " & \cite{Mendiratta-2014-Shifted-Lemn-Bernoulli}    \\
			\hline
			${R}_{\mathcal{S}^*_{e}}(\mathcal{S}^*)$ & $\frac{e-1}{e+1}$ & $0.462117$ & " & \cite{Mendiratta-Ravi-2015-Expo-BMMS}    \\
			\hline
			${R}_{\mathcal{S}^*_{\leftmoon}}(\mathcal{C})$ & $2-\sqrt{2}$ & $0.58578$ & $\frac{z}{1-z}$ & \cite{Gandhi-Ravi-2017-Lune}    \\
			\hline
			${R}_{\mathcal{S}^*_{e}}(\mathcal{C})$ & $\frac{e-1}{e}$ & $0.632121$ & " & \cite{Mendiratta-Ravi-2015-Expo-BMMS}    \\
			\hline
			${R}_{\mathcal{S}^*_{S}}(\mathcal{S}^*_{L})$ & $(2-\sin1)\sin1$ & 0.975 & $\frac{4z\exp(2\sqrt{1+z}-2)}{1+\sqrt{1+z}}$ & \cite{Cho-2019-Sine-BIMS}     \\
			\hline
			${R}_{\mathcal{S}^*_{R}}(\mathcal{S}^*_{L})$ & $8\sqrt{2}-11$ & 0.313708 & " & \cite{Kumar-Ravi-2016-Starlike-Associated-Rational-Function}     \\
			\hline
			${R}_{\mathcal{S}^*(\alpha)}(\mathcal{S}^*_{L})$ & $1-\alpha^2$ &  & " & \cite{Sokol-2009-Radius-Problems-SL*-AMC} \\
			\hline
			${R}_{\mathcal{S}^*_{\leftmoon}}(\mathcal{S}^*_{e})$ & $\log\left(\frac{\sqrt{2}+1}{2}\right)$ & 0.188226 & $z\exp\left(\int_0^z\frac{e^{t}-1}{t}\,dt\right)$ & \cite{Gandhi-Ravi-2017-Lune}     \\
			\hline
			${R}_{\mathcal{S}^*_{S}}(\mathcal{S}^*_{C})$ & {\scriptsize{$\frac{1}{2}\left(\sqrt{2(2+3\sin1)}-2\right)$}} & 0.504 & $z\exp\left(\frac{4z}{3}+\frac{z^2}{3}\right)$ & \cite{Cho-2019-Sine-BIMS}     \\
			\hline
			${R}_{\mathcal{S}^*_{L}}(\mathcal{S}^*_{C})$ & $-1+\sqrt{-\frac{1}{2}+\frac{3}{\sqrt{2}}}$ & 0.273311 & " & \cite{Sharma-Ravi-2016-Cardioid}     \\
			\hline
			${R}_{\mathcal{BS}^*(\alpha)}(\mathcal{S}^*_{S})$ & $\sinh^{-1}\left(\frac{1}{1+\alpha}\right)$ & & $z\exp\left(\int_0^z\frac{\sin{t}}{t}\,dt\right)$ &  \cite{Cho-Ravi-2018-Diff-Sub-Booth-Lem-TJM}    \\
			\hline
			${R}_{\mathcal{S}^*_{R}}(\mathcal{S}^*_{S})$ & {\scriptsize{$\log(3-2\sqrt{2})+\sqrt{6}(\sqrt{2}-1)$}} & 0.170742 & " &  \cite{Kumar-Ravi-2016-Starlike-Associated-Rational-Function}    \\
			\hline 	
		\end{tabular}
		\caption{Some recent radii results}
	\end{table}
	\label{table:II}
\end{center}
Motivated by the aforementioned works, in this work, we consider the function class $\mathcal{S}^*_{Ne}$ introduced by the authors in \cite{Wani-Swami-Nephroid-Basic} and find sharp $\mathcal{S}^*_{Ne}$-radii estimates for certain geometrically defined function classes available in the literature. More explicitly, in \Cref{Section-Radii-Results-Ma-Minda-Classes-RRN}, the $\mathcal{S}^*_{Ne}$-radius problem is discussed for some of the recently introduced Ma-Minda type starlike classes. In particular, the Janowski class $\mathcal{S}^*[A,B]$ is considered, and consequently, the best possible $\mathcal{S}^*_{Ne}$-radii for $\mathcal{S}^*$ and $\mathcal{C}$ are obtained as special cases. Finally, in \Cref{Section-Radii-Results-Ratio-Functions-RRN}, we consider several function families defined by the ratio of analytic functions and calculate $\mathcal{S}^*_{Ne,n}$-radius for them, where
$$\mathcal{S}^*_{Ne,n}:=\mathcal{S}^*_{Ne}\cap\mathcal{A}_{n},$$
and $\mathcal{A}_{n}$ consists of analytic functions of the form $f(z)=z+\sum_{k=n+1}^\infty a_kz^k$. Apart from proving the sharpness of radii estimates analytically, certain graphical illustrations are also provided.

The following lemma related to the largest disks contained in the region bounded by the nephroid curve \eqref{Equation-of-Nephroid} is instrumental in proving our results.
\begin{lemma}\label{Lemma-Radius-ra-NeP}
Let $1/3<a<5/3$, and let $r_a$  be given by
\begin{align*}
r_a=
\begin{cases}
a-{1}/{3} & \text{ if } ~ {1}/{3}<a\leq1,\\
{5}/{3}-a & \text{ if } ~ 1\leq{a}<{5}/{3}.
\end{cases}
\end{align*}
Then
\begin{align*}
\left\{w\in\mathbb{C}:|w-a|<r_a\right\}\subseteq\Omega_{Ne},
\end{align*}
where $\Omega_{Ne}:=\varphi_{\scriptscriptstyle{Ne}}(\mathbb{D})$, the region bounded by the nephroid \eqref{Equation-of-Nephroid}.
\end{lemma}
\begin{proof}
For $z=e^{it}$, the parametric equations of the nephroid $\varphi_{\scriptscriptstyle{Ne}}(z)=1+z-z^3/3$ are
\begin{align*}
u(t)=1+\cos{t}-(\cos{3t})/3, \quad v(t)=\sin{t}-(\sin{3t})/3, \quad t\in(-\pi,\pi].
\end{align*}
Let $a\in(1/3,5/3)$. Then the squared distance from the point $(a,0)$ on the real line to the points on the nephroid curve \eqref{Equation-of-Nephroid} is given by
\begin{align*}
\xi(t)&=(a-u(t))^2+(v(t))^2\\
     &=\frac{16}{9}+(a-1)^2-4(a-1)\cos{t}-\frac{4}{3}\cos^2{t}+\frac{8}{3}(a-1)\cos^3{t}\\
      &=\frac{16}{9}+(a-1)^2-4(a-1)x-\frac{4}{3}x^2+\frac{8}{3}(a-1)x^3=:H(x),
\end{align*}
where $x=\cos{t}\in[-1,1],\,-\pi<t\leq\pi$. Since the nephroid curve \eqref{Equation-of-Nephroid} is symmetric about the real line, it is enough to consider $t\in[0,\pi]$. A simple computation shows that $H'(x)=0$ if, and only if,
\begin{align*}
x=x_0:=\frac{1-\sqrt{1+18(a-1)^2}}{6(a-1)} ~ \text{ and } ~ x=x_1:=\frac{1+\sqrt{1+18(a-1)^2}}{6(a-1)}.
\end{align*}
For $a\in(1/3,5/3)$, further verification reveals (also see the plots in \Cref{Figures-Roots-RRN}) that only the number $x_0$ lies between $-1$ and $1$. Moreover,
\begin{align*}
H''(x_0)=-\frac{8}{3}\sqrt{1+18(a-1)^2}<0 ~\text{ for each } ~a.
\end{align*}
This shows that $x_0$ is the point of maxima for
the function $H(x)$ and hence, $H(x)$ is increasing in the interval
$[-1,x_0]$ and decreasing in $[x_0,1]$. Thus, for $1/3<a<5/3$, we have
\begin{align*}
\min_{0\leq{t}\leq\pi}{\xi(t)}=\min\left\{H(-1),\, H(1)\right\}.
\end{align*}
\begin{figure}[H]
\begin{subfigure}[b]{0.45\textwidth}
	\centering
 \includegraphics[scale=1.3]{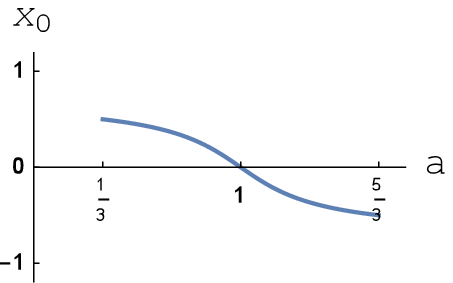}
\label{Figure-Root1-RRN}
\end{subfigure}
\begin{subfigure}[b]{0.45\textwidth}
	\centering
 \includegraphics[scale=1.3]{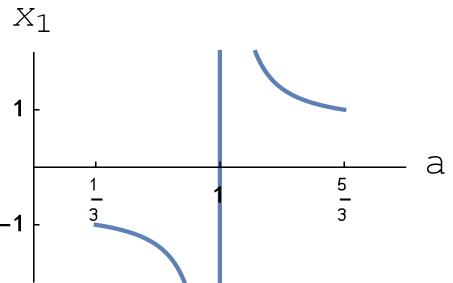}
\label{Figure-Root2-RRN}
\end{subfigure}
\caption{Plots of $x_0=x_0(a)$ and $x_1=x_1(a)$ for $a\in(1/3,5/3)$.}
\label{Figures-Roots-RRN}
\end{figure}

\noindent
Also,
$$H(1)-H(-1)=-{8(a-1)}/{3},$$
so that $H(-1)\leq{H(1)}$ whenever $a\leq1$, and $H(1)\leq{H(-1)}$ whenever $a\geq1$. Therefore, we conclude that
\begin{align*}
r_{a}=\min_{0\leq{t}\leq\pi}\sqrt{\xi(t)}=
\begin{cases}
\sqrt{H(-1)}=a-{1}/{3} &\text{ whenever } {1}/{3}<a\leq1,\\
\sqrt{H(1)}={5}/{3}-a   &\text{ whenever } 1\leq{a}<{5}/{3}.
\end{cases}
\end{align*}
Thus, the disk $\left\{w\in\mathbb{C}:|w-a|<r_a\right\}$ completely lies inside the region $\Omega_{Ne}$, and establishes the lemma.
\end{proof}

\section{$\mathcal{S}^*_{Ne}$-Radius For Ma-Minda Type Classes}
\label{Section-Radii-Results-Ma-Minda-Classes-RRN}
For $-1\leq{B}<A\leq{1}$, let $\mathcal{P}[A,B]$ be the collection of analytic functions $p:\mathbb{D}\to\mathbb{C}$ that are of  the form $p(z)=1+\sum_{n=1}^{\infty}p_nz^n$ and satisfy the subordination $p(z)\prec{(1+Az)/(1+Bz)}$.
Also, recall \cite{Janowski-1973-class-some-extremal} that the Janowski starlike class $\mathcal{S}^*[A,B]$ consists of functions $f\in\mathcal{A}$ satisfying $\mathcal{Q}_f(z)\in\mathcal{P}[A,B]$. To prove the results we use the following lemma.

\begin{lemma}[{\cite[Lemma 2.1, p. 267]{Ravi-Ronning-1997-Complex-Variables}}]\label{Lemma2.1-Ravi-Ronning-1997-P{A,B}}
If $p\in\mathcal{P}[A,B]$, then for $|z|=r<1$,
\begin{align*}
\left|p(z)-\frac{1-ABr^{2}}{1-B^2r^2}\right|\leq\frac{(A-B)r}{1-B^2r^2}.
\end{align*}
\end{lemma}

\begin{theorem}\label{Theorem-Janowski-Class-RRN}
	\begin{enumerate}[{\rm(i)}]
\item
	Let $0\leq{B}<A\leq1$. Then the $\mathcal{S}^*_{Ne}$-radius for $\mathcal{S}^*[A,B]$ is given by
	\begin{align*}
	{R}_{\mathcal{S}^*_{Ne}}\left(\mathcal{S}^*[A,B]\right)=\min\left\{1,\frac{2}{3A-B}\right\}.
	\end{align*}
	In particular, if $1-B\leq3(1-A)$, then $\mathcal{S}^*[A,B]\subset\mathcal{S}^*_{Ne}$.
\item
	Let $-1\leq{B}<A\leq1$ with $B\leq0$. Then the  $\mathcal{S}^*_{Ne}$-radius for $\mathcal{S}^*[A,B]$ is given by
	\begin{align*}
	{R}_{\mathcal{S}^*_{Ne}}\left(\mathcal{S}^*[A,B]\right)=\min\left\{1,\frac{2}{3A-5B}\right\}.
	\end{align*}
	In particular, if $3(1+A)\leq5(1+B)$, then $\mathcal{S}^*[A,B]\subset\mathcal{S}^*_{Ne}$.
	\end{enumerate}
\end{theorem}
\begin{proof}
	\begin{enumerate}[(i):]
\item
	Let $f\in\mathcal{S}^*[A,B]$. Then $\mathcal{Q}_f\in\mathcal{P}[A,B]$, and \Cref{Lemma2.1-Ravi-Ronning-1997-P{A,B}} yields
	\begin{align}\label{Disk-Ravi-Ronning-P{A,B}-RRN}
	\left|\mathcal{Q}_f(z)-\frac{1-ABr^{2}}{1-B^2r^2}\right|\leq\frac{(A-B)r}{1-B^2r^2}.
	\end{align}
	The inequality \eqref{Disk-Ravi-Ronning-P{A,B}-RRN} represents a disk with center at ${(1-ABr^{2})}/{(1-B^2r^2)}$ and radius ${(A-B)r}/{(1-B^2r^2)}$. Since $B\geq0$,
	we have $(1-ABr^2)/(1-B^2r^2)\leq1$. Thus, by \Cref{Lemma-Radius-ra-NeP}, the disk \eqref{Disk-Ravi-Ronning-P{A,B}-RRN} lies completely inside the region $\Omega_{Ne}$ bounded by the nephroid curve \eqref{Equation-of-Nephroid} if
	\begin{align*}
	\frac{(A-B)r}{1-B^2r^2}\leq \frac{1-ABr^{2}}{1-B^2r^2}-\frac{1}{3}.
	\end{align*}
	Simplifying the above inequality, we obtain $r\leq2/(3A-B)=:\rho_{\scriptscriptstyle{J}}$. For sharpness of the estimate, consider the function
	\begin{align*}
	f_{\scriptscriptstyle{J}}(z):=
	\begin{cases}
	z\left(1+Bz\right)^{\frac{A-B}{B}}, & \text{ if } ~    B\neq0\\
	ze^{Az},                               &  \text{ if } ~  B=0.
	\end{cases}
	\end{align*}
	The function $f_{\scriptscriptstyle{J}}(z)$ satisfies $\mathcal{Q}_{f_{\scriptscriptstyle{J}}}(z)={(1+Az)}/{(1+Bz)}$, implying that $f_{\scriptscriptstyle{J}}\in\mathcal{S}^*[A,B]$. Also, for $z_0=-{2}/{(3A-B)}\in\partial\mathbb{D}_{\rho_{\scriptscriptstyle{J}}}$, we have
	\begin{align*}
	\mathcal{Q}_{f_{\scriptscriptstyle{J}}}(z_0)=1/3\in\partial\Omega_{Ne}
	\end{align*}
    i.e.,
    \begin{align*}
    \mathcal{Q}_{f_J}(\partial\mathbb{D}_{\rho_J})\cap\partial\Omega_{Ne}=\{1/3\}\neq\emptyset.
    \end{align*}	
    This proves that the result is sharp and   $f_{\scriptscriptstyle{J}}\in\mathcal{S}^*[A,B]$ is the extremal function.\\
     In particular, if $1-B\leq3(1-A)$, then $2/(3A-B)\geq1$ and hence, ${R}_{\mathcal{S}^*_{Ne}}\left(\mathcal{S}^*[A,B]\right)=1$. That is, $\mathcal{S}^*[A,B]\subset\mathcal{S}^*_{Ne}$.	
\item
	As afore, $f\in\mathcal{S}^*[A,B]$ implies the inequality \eqref{Disk-Ravi-Ronning-P{A,B}-RRN}.
	Now, $B\leq0$ implies that
	 $(1-ABr^2)/(1-B^2r^2)\geq1$. Therefore, in this case, the disk \eqref{Disk-Ravi-Ronning-P{A,B}-RRN} lies in the interior of $\Omega_{Ne}$
	if
	\begin{align*}
	\frac{(A-B)r}{1-B^2r^2}\leq \frac{5}{3}-\frac{1-ABr^{2}}{1-B^2r^2}.
	\end{align*}
	Solving, we get $r\leq2/(3A-5B)$. Choosing $z_0=2/(3A-5B)$ and $f_{\scriptscriptstyle{J}}\in\mathcal{S}^*[A,B]$ defined above, it can be easily verified that
	$\mathcal{Q}_{f_{\scriptscriptstyle{J}}}(z_0)=5/3\in\partial\Omega_{Ne}$.
		Hence, the result is sharp.
\qedhere
	\end{enumerate}
\end{proof}
Let $0\leq{B}<A\leq1$. Using a result proved by Wani and Swaminathan \cite{Wani-Swami-Nephroid-Basic}, the following remark also shows that $1-B\leq3(1-A)$ is indeed a sufficient condition for the inclusion $\mathcal{S}^*[A,B]\subset\mathcal{S}^*_{Ne}$.

\begin{remark}
	Since, $1-B\leq3(1-A)$ implies $1-B^2 \leq 3(1-AB+(B-A))<3(1-AB)$ and,
	$0\leq{B}<A\leq1$ gives $3(1-AB)\leq3(1-B^2)$. Combining, we obtain $1-B^2<3(1-AB)\leq3(1-B^2)$. Now the inclusion relation  $\mathcal{S}^*[A,B]\subset\mathcal{S}^*_{Ne}$ follows from {\rm\cite[Theorem 3.3]{Wani-Swami-Nephroid-Basic}}.
\end{remark}

Specializing the constants $A$ and $B$ in \Cref{Theorem-Janowski-Class-RRN}, the following important sharp radii results are obtained.

\begin{corollary}
The sharp $\mathcal{S}^*_{Ne}$-radius for $\mathcal{S}^*(\alpha)=\mathcal{S}^*[1-2\alpha,-1]$ is ${2}/{(3(1-2\alpha)+5)}$. The estimate is sharp for the function $k_{\alpha}(z)=z(1-z)^{2\alpha-2}$, $0\leq\alpha<1$.
\end{corollary}

\begin{corollary}
The sharp $\mathcal{S}^*_{Ne}$-radius for the class of starlike functions $\mathcal{S}^*$ is ${1}/{4}$, and the sharpness holds for the well-known Koebe function $k(z):=z/(1-z)^{2}$.
\end{corollary}

\begin{corollary}
The sharp $\mathcal{S}^*_{Ne}$-radius for the convex class $\mathcal{C}$ is ${2}/{5}$.
\end{corollary}
\begin{proof}
From \Cref{Theorem-Janowski-Class-RRN}, the $\mathcal{S}^*_{Ne}$-radius for $\mathcal{S}^*\left({1}/{2}\right)=\mathcal{S}^*[0,-1]$ is ${2}/{5}$. Using the fact that $\mathcal{C}\subset\mathcal{S}^*\left({1}/{2}\right)$ (see \cite[Theorem 2.6a, p. 57]{Miller-Mocanu-Book-2000-Diff-Sub}), it follows that ${R}_{\mathcal{S}^*_{Ne}}(\mathcal{C})\geq2/5$. Further, at the point $z_0=2/5\in\partial\mathbb{D}_{2/5}$, the function $h(z):=z/(1-z)\in\mathcal{C}$ satisfies
\begin{align*}
\mathcal{Q}_h(z_0)={1}/{(1-2/5)}={5}/{3}\in\partial\Omega_{Ne},
\end{align*}
showing that ${R}_{\mathcal{S}^*_{Ne}}(\mathcal{C})\leq2/5$. Hence, ${R}_{\mathcal{S}^*_{Ne}}(\mathcal{C})=2/5$, and the estimate is sharp for the function $h(z)$.
\end{proof}

Before proving the next result, we note that $\mathcal{S}^*_{L}(\alpha)\subset\mathcal{S}^*_{Ne}$ for $\alpha\geq1/3$, and $\mathcal{S}^*_{\alpha,e}\subset\mathcal{S}^*_{Ne}$ for $\alpha\geq{(3e-5)}/{(3e-3)}$, (see Wani and Swaminathan \cite[Theorem 3.1]{Wani-Swami-Nephroid-Basic}). Therefore
\begin{align*}
{R}_{\mathcal{S}^*_{Ne}}\left(\mathcal{S}^*_{L}(\alpha)\right)=1,  \mbox{ whenever } \alpha\geq1/3
\end{align*}
and
\begin{align*}
{R}_{\mathcal{S}^*_{Ne}}\left(\mathcal{S}^*_{\alpha,e}\right)=1,  \mbox{ whenever } \alpha\geq{(3e-5)}/{(3e-3)}.
\end{align*}

\begin{theorem}
	For the function classes
	$\mathcal{BS}^*(\alpha)$,
	$\mathcal{S}^*_{L}(\alpha)$,
	and
	$\mathcal{S}^*_{\alpha,e}$,
	the following radius results hold:
\begin{enumerate}[{\rm(i)}]
\item
	For $\alpha\in[0,1)$, ${R}_{\mathcal{S}^*_{Ne}}\left(\mathcal{BS}^*(\alpha)\right)=\rho_{\scriptscriptstyle{\mathcal{B}}}(\alpha):={4}/{(3+\sqrt{9+16\alpha})}$.
\item
	For $\alpha\in[0,1/3]$, ${R}_{\mathcal{S}^*_{Ne}}\left(\mathcal{S}^*_{L}(\alpha)\right)=\rho_{\scriptscriptstyle{L}}(\alpha):={4(2-3\alpha)}/{9(1-\alpha)^2}$. In particular, ${R}_{\mathcal{S}^*_{Ne}}\left(\mathcal{S}^*_{L}\right)=\rho_{\scriptscriptstyle{L}}:={8}/{9}$.
\item
	For $\alpha\in[0,(3e-5)/(3e-3)]$,
	${R}_{\mathcal{S}^*_{Ne}}\left(\mathcal{S}^*_{\alpha,e}\right)=\rho_{\scriptscriptstyle{e}}(\alpha):=
	\log\left({(5-3\alpha)}/{(3-3\alpha)}\right)$.
	In particular, ${R}_{\mathcal{S}^*_{Ne}}\left(\mathcal{S}^*_{e}\right)=\rho_{\scriptscriptstyle{e}}:=\log\left({5}/{3}\right)\approx0.510826$.

\end{enumerate}
Each of the estimates is best possible.
\end{theorem}
\begin{proof}
	\begin{enumerate}[(i):]
\item
		Let $f\in\mathcal{BS}^*(\alpha)$. Then $\mathcal{Q}_{f}(z)\prec G_{\alpha}(z):=1+z/(1-\alpha z^2)$ and hence, for $|z|=r$, we have
		\begin{align} \label{Disk-Booth-Lem-RRN}
		\left|\mathcal{Q}_{f}(z)-1\right|\leq \left|{z}/{(1-\alpha z^2)}\right|\leq {r}/{(1-\alpha r^2)}.
		\end{align}
		In view of \Cref{Lemma-Radius-ra-NeP}, the disk \eqref{Disk-Booth-Lem-RRN} lies completely inside the nephroid region $\Omega_{Ne}$ if
		$r/(1-\alpha r^2)\leq2/3$. This gives $r\leq\rho_{\scriptscriptstyle{\mathcal{B}}}(\alpha)$. For sharpness, consider the function
		\begin{align*}
		f_{\scriptscriptstyle{\mathcal{B}}}(z)=
		\begin{cases}
	       z \left(\frac{1+\sqrt{\alpha}z}{1-\sqrt{\alpha}z}\right)^{1/(2\sqrt{\alpha})},& \quad \alpha\in(0,1) \\
	       z e^z,    &  \quad \alpha=0.
		\end{cases}
		\end{align*}
		It is easy to verify that $\mathcal{Q}_{f_{\scriptscriptstyle{\mathcal{B}}}}(z)=G_{\alpha}(z)$, and hence $f_{\scriptscriptstyle{\mathcal{B}}}\in\mathcal{BS}^*(\alpha)$. Also, a straightforward calculation shows that
		\begin{align*}
		\mathcal{Q}_{f_{\scriptscriptstyle{\mathcal{B}}}}(z_0)=1/3\in\partial\Omega_{Ne}, \text{ for } z_0=-\rho_{\scriptscriptstyle{\mathcal{B}}}(\alpha),
		\end{align*}
		and
		\begin{align*}
		\mathcal{Q}_{f_{\scriptscriptstyle{\mathcal{B}}}}(z_1)=5/3\in\partial\Omega_{Ne}, \text{ for } z_1=\rho_{\scriptscriptstyle{\mathcal{B}}}(\alpha),
		\end{align*}
\begin{figure}[H]
	\begin{subfigure}{0.45\textwidth}
		\centering
		\includegraphics[scale=.75]{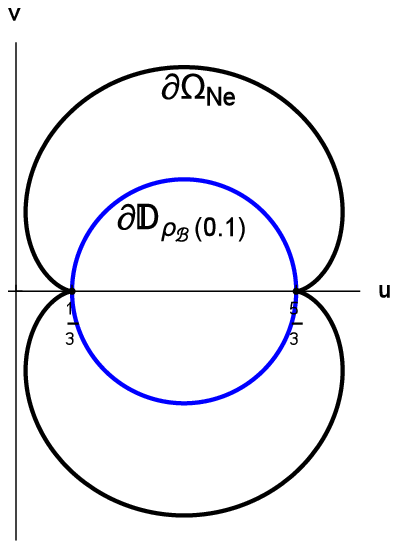}
		\caption{}
	\end{subfigure}
	\begin{subfigure}{0.45\textwidth}
		\centering
		\includegraphics[scale=.75]{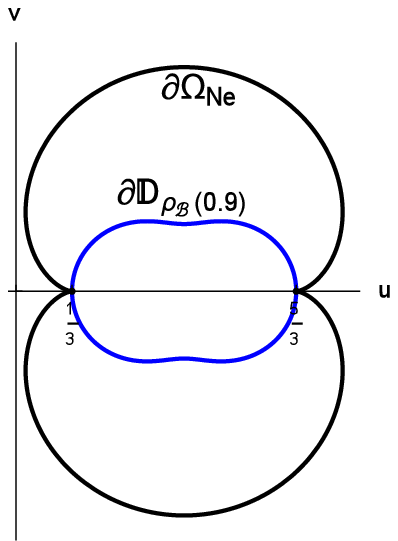}
		\caption{}
	\end{subfigure}
	\caption{Sharpness of $\rho_{\scriptscriptstyle{\mathcal{B}}}(\alpha)$, where (a) $\alpha=0.1$ and (b) $\alpha=0.9$}
	\label{Figure-Sharpness-Booth-Lem-RRN}
\end{figure}
Therefore,
\begin{align*}
\mathcal{Q}_{f_{\scriptscriptstyle{\mathcal{B}}}}\left(\partial\mathbb{D}_{\rho_{\scriptscriptstyle{\mathcal{B}}}(\alpha)}\right)\cap\partial\Omega_{Ne}=\{1/3,5/3\}\neq\emptyset.
\end{align*}		
This proves that the estimate $\rho_{\scriptscriptstyle{\mathcal{B}}}(\alpha)$ is best possible.  \Cref{Figure-Sharpness-Booth-Lem-RRN} depicts the sharpness of $\rho_{\scriptscriptstyle{\mathcal{B}}}(\alpha)$ for different values of $\alpha$.		
\item
		Let $f\in\mathcal{S}^*_L(\alpha)$. Then $\mathcal{Q}_{f}(z)\prec\alpha+(1-\alpha)\sqrt{1+z}$ and hence
		\begin{align*}
		\left|\mathcal{Q}_{f}(z)-1\right| &\leq \left|\alpha+(1-\alpha)\sqrt{1+z}-1\right|\\
		&= (1-\alpha)\left|1-\sqrt{1+z}\right|\\
		&\leq (1-\alpha)\left(1-\sqrt{1-r}\right),  \quad |z|=r.
		\end{align*}
		An application of \Cref{Lemma-Radius-ra-NeP} shows that $f\in\mathcal{S}^*_{Ne}$ if $(1-\alpha)\left(1-\sqrt{1-r}\right)\leq2/3$, which on simplification gives $r\leq\rho_{\scriptscriptstyle{L}}(\alpha)$. Verify that the function
		\begin{align*}
		f_{\scriptscriptstyle{L}}(z)=z+(1-\alpha)z^2+\frac{1}{16}(1-\alpha)(1-2\alpha)z^3+\cdots
		\end{align*}
		satisfies $\mathcal{Q}_{f_{\scriptscriptstyle{L}}}(z)=\alpha+(1-\alpha)\sqrt{1+z}$, and hence is a member of $\mathcal{S}_L^*(\alpha)$. Moreover,
			\begin{align*}
		\mathcal{Q}_{f_{\scriptscriptstyle{L}}}(z_0)=1/3\in\partial\Omega_{Ne}, ~~ \text{ for } z_0=-\rho_{\scriptscriptstyle{L}}(\alpha)\in\partial\mathbb{D}_{\rho_{\scriptscriptstyle{L}}(\alpha)}.
		\end{align*}
\begin{figure}[H]
	\begin{subfigure}{0.45\textwidth}
		\centering
		\includegraphics[scale=1.5]{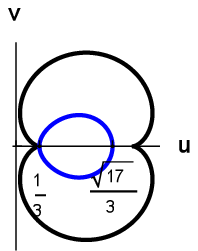}
		\caption{$\alpha=0$}
	\end{subfigure}
	\begin{subfigure}{0.45\textwidth}
		\centering
		\includegraphics[scale=1.5]{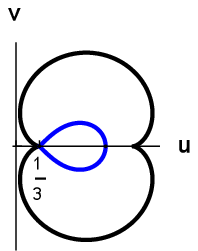}
		\caption{$\alpha=1/3$}
	\end{subfigure}
	\caption{Sharpness of $\rho_{\scriptscriptstyle{L}}(\alpha)$}
	\label{Figure-Sharpness-GenLem-RRN}
\end{figure}
\item
		$f\in\mathcal{S}^*_{\alpha,e}$ implies $\mathcal{Q}_{f}(z)\prec\alpha+(1-\alpha)e^z$. For $|z|=r$, this subordination yields
		\begin{align*}
		\left|\mathcal{Q}_{f}(z)-1\right| 
		\leq (1-\alpha)\left|e^z-1\right|
		\leq (1-\alpha)\left( e^r-1 \right)
		\leq {2}/{3}
		\end{align*}
		if $r\leq\log\left({(5-3\alpha)}/{(3-3\alpha)}\right)=\rho_{\scriptscriptstyle{e}}(\alpha)$. Therefore, $f\in\mathcal{S}^*_{Ne}$ if $|z|=r\leq\rho_{\scriptscriptstyle{e}}(\alpha)$. The result is sharp for the function $f_{\scriptscriptstyle{e}}\in\mathcal{S}^*_{\alpha,e}$ given by
		\begin{align*}
		f_{\scriptscriptstyle{e}}(z)=z+(1-\alpha)z^2+\frac{1}{4}(1-\alpha)(3-2\alpha)z^3+\cdots,
		\end{align*}
		and
		\begin{align*}
		 \mathcal{Q}_{f_{\scriptscriptstyle{e}}}\left(\partial\mathbb{D}_{\rho_{\scriptscriptstyle{e}}(\alpha)}\right)\cap\partial\Omega_{Ne}=\{5/3\}.
		\end{align*}
	\end{enumerate}
 \begin{figure}[H]
 	\begin{subfigure}{0.45\textwidth}
 		\centering
 		\includegraphics[scale=1.5]{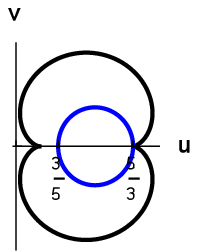}
 		\caption{$\alpha=0$}
 	\end{subfigure}
 	\begin{subfigure}{0.45\textwidth}
 		\centering
 		\includegraphics[scale=1.5]{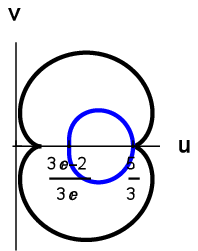}
 		\caption{$\alpha=(3e-5)/(3e-3)$}
 	\end{subfigure}
 	\caption{Sharpness of $\rho_{\scriptscriptstyle{e}}(\alpha)$}
 	\label{Figure-Sharpness-GenExpo-RRN}
 	\qedhere
 \end{figure}
\end{proof}

\begin{theorem}
	For the function classes
$\mathcal{S}^*_{RL}$,
$\mathcal{S}^*_{C}$
and
$\mathcal{S}^*_{R}$,
 we have:
\begin{enumerate}[{\rm(i)}]
	\item
	${R}_{\mathcal{S}^*_{Ne}}\left(\mathcal{S}^*_{RL}\right)=\rho_{\scriptscriptstyle{RL}}:={56}/{(122-41 \sqrt{2})}\approx0.874764$,
	\item
	${R}_{\mathcal{S}^*_{Ne}}\left(\mathcal{S}^*_{C}\right)=\rho_{\scriptscriptstyle{c}}:=\sqrt{2}-1 \approx 0.414214$,
	\item
	 ${R}_{\mathcal{S}^*_{Ne}}\left(\mathcal{S}^*_{R}\right)=\rho_{\scriptscriptstyle{R}}:={1}/{(3\sqrt{2}-3)}\approx0.804738$.
\end{enumerate}
The estimate in each part is sharp (see {\rm\Cref{Figure-Sharpness-SCR-RRN}}).
\end{theorem}
\begin{proof}
	\begin{enumerate}[(i):]
 \item
		Let $f\in\mathcal{S}^*_{RL}$. Then $\mathcal{Q}_{f}(z)\prec\varphi_{\scriptscriptstyle{RL}}(z)$, where $\varphi_{\scriptscriptstyle{RL}}(z)$ is defined in \eqref{Definition-Shifted-Lem-Ber}.
		Thus, for $|z|=r<1$, we have
		\begin{align*}
		\left|\mathcal{Q}_{f}(z)-1\right| & \leq \left| \sqrt{2}-(\sqrt{2}-1)\sqrt{\frac{1-z}{1+2(\sqrt{2}-1)z}}-1 \right| \\
		& \leq  1- \left( \sqrt{2}-(\sqrt{2}-1)\sqrt{\frac{1+r}{1-2(\sqrt{2}-1)z}} \right).
		\end{align*}
		In view of \Cref{Lemma-Radius-ra-NeP}, the above disk lies inside the nephroid region $\Omega_{Ne}$ provided
		\begin{align*}
		1- \left( \sqrt{2}-(\sqrt{2}-1)\sqrt{\frac{1+r}{1-2(\sqrt{2}-1)z}} \right) \leq \frac{2}{3},
		\end{align*}
		Or equivalently, if $r\leq \rho_{\scriptscriptstyle{RL}}$. The result is sharp for the function $f_{\scriptscriptstyle{RL}}\in\mathcal{S}^*_{RL}$ given by
			\begin{align*}
		f_{\scriptscriptstyle{RL}}(z)=z &\left( \frac{ \sqrt{1-z}+\sqrt{1+2(\sqrt{2}-1)z}} {2}  \right)^{2\sqrt{2}-2}\\
		 & \hspace{8em} \times \exp\left(\sqrt{2 \left(\sqrt{2}-1\right)}\;
		\tan^{-1}\left(\Psi(z)\right)\right),
		\end{align*}
		where
		\begin{align*}
		\Psi(z)=\frac{\sqrt{2 \left(\sqrt{2}-1\right)} \left(\sqrt{2 \left(\sqrt{2}-1\right) z+1}-\sqrt{1-z}\right)}
		{2 \left(\sqrt{2}-1\right) \sqrt{1-z}+\sqrt{2 \left(\sqrt{2}-1\right) z+1}}.
		\end{align*}
		Further,
		 $\mathcal{Q}_{f_{\scriptscriptstyle{RL}}}\left(\partial\mathbb{D}_{\rho_{\scriptscriptstyle{RL}}}\right)\cap\partial\Omega_{Ne}=\{1/3\}$.
 \item
		$f\in\mathcal{S}^*_C$ implies the subordination $\mathcal{Q}_{f}(z) \prec 1+4z/3+2z^2/3$, which in turn gives
		\begin{align*}
		\left|\mathcal{Q}_{f}(z)-1\right| \leq
		2\left( r^2+2r \right)/3, \qquad |z|=r.
		\end{align*}
		Applying \Cref{Lemma-Radius-ra-NeP}, we conclude that
		$f\in\mathcal{S}^*_{Ne}$ if $r^2+2r\leq1$, or, if $r\leq \sqrt{2}-1=\rho_{\scriptscriptstyle{c}}$. To verify the sharpness of the radius estimate $\rho_{\scriptscriptstyle{c}}$, consider the function
		\begin{align*}
		f_{\scriptscriptstyle{c}}(z)=z \exp\left(\frac{4}{3}z+\frac{1}{3}z^2\right).
		\end{align*}
		Observe that $f_{\scriptscriptstyle{c}}\in\mathcal{S}^*_C$, and
		\begin{align*}
		\mathcal{Q}_{f_{\scriptscriptstyle{c}}}(z_0)=5/3\in\partial\Omega_{Ne}, ~~ \text{ for } z_0=\sqrt{2}-1\in\partial\mathbb{D}_{\rho_{\scriptscriptstyle{c}}}.
		\end{align*}
		 Hence the result is sharp.
 \item
		Let $f\in\mathcal{S}^*_{R}$. Then $\mathcal{Q}_{f}(z)-1\prec {z(k+z)}/{k(k-z)}$, where $k=\sqrt{2}+1$. For $|z|=r$, this relation further implies
		\begin{align*}
		\left|\mathcal{Q}_{f}(z)-1\right| \leq
		{r(k+r)}/{k(k-r)} \leq {2}/{3}
		\end{align*}
		if $3r^2+5kr-2k^2\leq0$. This yields $r\leq\rho_{\scriptscriptstyle{R}}$. Thus, in view of \Cref{Lemma-Radius-ra-NeP}, $f\in\mathcal{S}^*_{Ne}$ for $|z|\leq\rho_{\scriptscriptstyle{R}}$. Now consider the function
		\begin{align*}
		f_{\scriptscriptstyle{R}}(z)= \frac{k^2 z}{(k-z)^2} e^{-z/k}, \qquad k=\sqrt{2}+1.
		\end{align*}
		Simple calculations show that $\mathcal{Q}_{f_{\scriptscriptstyle{R}}}(z)=(k^2+z^2)/k(k-z)$, and hence $f_{\scriptscriptstyle{R}}\in\mathcal{S}^*_{R}$. Moreover, the value of $\mathcal{Q}_{f_{\scriptscriptstyle{R}}}(z)$ at the point $z=\rho_{\scriptscriptstyle{R}}\in\partial\mathbb{D}_{\rho_{\scriptscriptstyle{R}}}$ is $5/3$, i.e.,
		\begin{align*}
		 \mathcal{Q}_{f_{\scriptscriptstyle{R}}}\left(\partial\mathbb{D}_{\rho_{\scriptscriptstyle{R}}}\right)\cap\partial\Omega_{Ne}=\{5/3\}\neq\emptyset.
		\end{align*}
		Hence the estimate is best possible.
\qedhere
	\end{enumerate}
	\end{proof}
\begin{figure}[H]
	\begin{subfigure}{0.3\textwidth}
		\centering
		\includegraphics[scale=1.5]{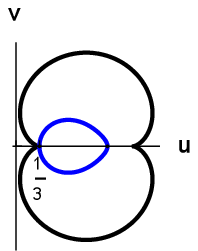}
		\caption{Sharpness of $\rho_{\scriptscriptstyle{RL}}$}
	\end{subfigure}
	\begin{subfigure}{0.3\textwidth}
		\centering
		\includegraphics[scale=1.5]{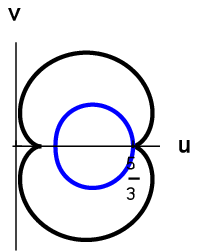}
		\caption{Sharpness of $\rho_{\scriptscriptstyle{c}}$}
	\end{subfigure}
    \begin{subfigure}{0.3\textwidth}
    	\centering
    	\includegraphics[scale=1.5]{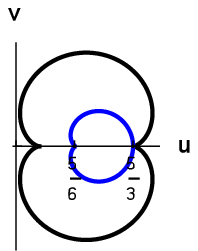}
    	\caption{Sharpness of $\rho_{\scriptscriptstyle{R}}$}
    \end{subfigure}
	\caption{}
	\label{Figure-Sharpness-SCR-RRN}
	\qedhere
\end{figure}

\begin{theorem}
The $\mathcal{S}^*_{Ne}$-radii for the classes $\mathcal{S}^*_{\leftmoon}$ and $\mathcal{S}^*_{S}$ are given by
\begin{enumerate}[{\rm(i)}]
	\item
	${R}_{\mathcal{S}^*_{Ne}}\left(\mathcal{S}^*_{\leftmoon}\right)=(\sqrt{17}-1)/6 \approx 0.520518$,
	\item
	${R}_{\mathcal{S}^*_{Ne}}\left(\mathcal{S}^*_{S}\right)=\sinh^{-1}(2/3)=\log\left((\sqrt{13}+2)/3\right) \approx 0.625145$.
\end{enumerate}
The estimates are not sharp.
\end{theorem}
\begin{proof}
  \begin{enumerate}[(i):]
\item
  	Let $f\in\mathcal{S}^*_{\leftmoon}$. Then $\mathcal{Q}_{f}(z)\prec\varphi_{\scriptscriptstyle{\leftmoon}}(z):= z+\sqrt{1+z^2}$, and hence
  	\begin{align} \label{Disk-Leftmoon-RRN}
  	\left|\mathcal{Q}_{f}(z)-1\right| \leq  \left| z+\sqrt{1+z^2}-1 \right|
  	\leq  1+r-\sqrt{1-r^2},    \quad |z|=r<1.
  	\end{align}
  	From \Cref{Lemma-Radius-ra-NeP}, it follows that the disk \eqref{Disk-Leftmoon-RRN} lies in the interior of the nephroid region  $\Omega_{Ne}$ if $1+r-\sqrt{1-r^2}\leq 2/3$, or, if $r^2+{r}/{3}-{4}/{9}\leq0$. This inequality gives the desired radius estimate. We note that a graphical observation, as well as the fact
  	\begin{align*}
  	\varphi_{\scriptscriptstyle{\leftmoon}}\left(|z|=(\sqrt{17}-1)/6\right)\cap\partial\Omega_{Ne}=\emptyset,
  	\end{align*}
  	 shows that the estimate is not sharp.
\item
  	If $f\in\mathcal{S}^*_S$, then $\mathcal{Q}_{f}(z)\prec 1+\sin{z}$. For, $|z|=r<1$, this yields
  	\begin{align*}
  	\left|\mathcal{Q}_{f}(z)-1\right| \leq \left| \sin{z} \right| \leq \sinh{r} \leq {2}/{3}
  	\end{align*}
  	provided $r\leq\sinh^{-1}(2/3)$. The desired result follows by an application of \Cref{Lemma-Radius-ra-NeP}. Again, the result is non-sharp.
\qedhere
\end{enumerate}
\end{proof}
In view of the above theorem, the following problem is posed.
\begin{problem}
  To find the sharp $\mathcal{S}^*_{Ne}$-radii for the families $\mathcal{S}^*_{\leftmoon}$ and $\mathcal{S}^*_S$.
\end{problem}


\section{$\mathcal{S}^*_{Ne,n}$-Radius For Classes Defined as Ratio of Functions} \label{Section-Radii-Results-Ratio-Functions-RRN}
Let $n\in\mathbb{N}:=\{1,2,3,...\}$. Let $\mathcal{A}_n$ be the collection of all analytic functions $f(z)$ of the form
\begin{align*}
f(z)=z+\sum_{k=n+1}^\infty a_kz^k.
\end{align*}
 For $0\leq\alpha<1$, let $\mathcal{P}_n(\alpha)$ consists of analytic functions $p(z)=1+\sum_{k=n}^\infty p_kz^k$ satisfying $\mathrm{Re}(p(z))>\alpha$, $z\in\mathbb{D}$. Further, let
\begin{align*}
\mathcal{S}^*_{Ne,n}:=\mathcal{S}^*_{Ne}\cap\mathcal{A}_n, \;
\mathcal{S}^*_n:=\mathcal{S}^*\cap\mathcal{A}_n, \;
\mathcal{C}_n:=\mathcal{C}\cap\mathcal{A}_n
\text{ and }
\mathcal{P}_n:=\mathcal{P}_n(0).
\end{align*}
In this section, we will find the $\mathcal{S}^*_{Ne,n}$-radius for several families defined as the ratio of certain special analytic functions. In the first theorem, we consider the following four families. For details , we refer \cite{Ali-Jain-Ravi-2013-Radius-Constants-BMMS}.
    \begin{align*}
    \mathcal{G}_1:=
    \left\{f\in\mathcal{A}_n:\frac{f}{g}\in\mathcal{P}_n \text{ and } \frac{g(z)}{z}\in\mathcal{P}_n,\; g\in\mathcal{A}_n\right\},
    \end{align*}

    \begin{align*}
    \mathcal{G}_2:=
    \left\{f\in\mathcal{A}_n:\frac{f}{g}\in\mathcal{P}_n \text{ and } \frac{g(z)}{z}\in\mathcal{P}_n\left(\frac{1}{2}\right),\; g\in\mathcal{A}_n\right\},
    \end{align*}

    \begin{align*}
    \mathcal{G}_3:=
    \left\{f\in\mathcal{A}_n:\left|\frac{f(z)}{g(z)}-1\right|<1 \text{ and } \frac{g(z)}{z}\in\mathcal{P}_n,\; g\in\mathcal{A}_n\right\},
    \end{align*}
    and
    \begin{align*}
    \mathcal{G}_4:=
    \left\{f\in\mathcal{A}_n:\left|\frac{f(z)}{g(z)}-1\right|<1, \; g\in\mathcal{C}_n\right\}.
    \end{align*}
To prove the results, we make use of the following combined lemma.


\begin{lemma}[{\cite{Ravi-Ronning-1997-Complex-Variables},\cite{Shah-1972-Pacific-JM}}] \label{Lemmas-Combined-Ravi-1997-Shah-1972}
	Let $p\in\mathcal{P}_n(\alpha)$. Then, for $|z|=r$, the following results hold.
	\begin{enumerate}[{\rm(a)}]
\item \label{Lemmas-Combined-a-Ravi}
{\rm (cf. \cite[Lemma 2.1]{Ravi-Ronning-1997-Complex-Variables})}
		\begin{align*}
		\left|p(z)-\frac{1+(1-2{\alpha})r^{2n}}{1-r^{2n}}\right|\leq\frac{2(1-\alpha)r^n}{1-r^{2n}}.
		\end{align*}
\item \label{Lemmas-Combined-b-Shah}
{\rm (cf. \cite[Lemma 2]{Shah-1972-Pacific-JM}) }
	\begin{align*}
	\left|\frac{zp'(z)}{p(z)}\right| \leq\frac{2nr^n(1-\alpha)}{(1-r^n)(1+(1-2\alpha)r^n)}.
	\end{align*}
	\end{enumerate}
\end{lemma}

\begin{theorem}
For the function classes
$\mathcal{G}_1$,
$\mathcal{G}_2$,
$\mathcal{G}_3$
and
$\mathcal{G}_4$,
we have the following radii results.
\begin{enumerate}[{\rm(i)}]
\item
${R}_{\mathcal{S}^*_{Ne,n}}\left(\mathcal{G}_1\right)=\rho_{\scriptscriptstyle{1}}:=\left(3n+\sqrt{9n^2+1}\right)^{-1/n}$.
\item
${R}_{\mathcal{S}^*_{Ne,n}}\left(\mathcal{G}_2\right)=\rho_{\scriptscriptstyle{2}}:=\sqrt[n]{4} \times \left(9n+\sqrt{81n^2+24n+16}\right)^{-1/n}$.
\item
${R}_{\mathcal{S}^*_{Ne,n}}\left(\mathcal{G}_3\right)=\rho_{\scriptscriptstyle{3}}:=\sqrt[n]{4} \times \left(9n+\sqrt{(4+9n)^2-48n}\right)^{-1/n}$.
\item
${R}_{\mathcal{S}^*_{Ne,n}}\left(\mathcal{G}_4\right)=\rho_{\scriptscriptstyle{4}}:=\sqrt[n]{4} \times \left(3(n+1)+\sqrt{(1+3n)^2+36n}\right)^{-1/n}$.
\end{enumerate}
Each estimate is sharp.
\end{theorem}
\begin{proof}
\begin{enumerate}[{(i):}]
\item
	Let $f\in\mathcal{G}_1$. Further, suppose that
	\begin{align*}
	p(z)={g(z)}/{z} \quad \text{ and } \quad q(z)={f(z)}/{g(z)}.
	\end{align*}
	Clearly, $p,q\in\mathcal{P}_n$. Thus, for $|z|=r$, the following inequalities follow from \Cref{Lemmas-Combined-Ravi-1997-Shah-1972}\eqref{Lemmas-Combined-b-Shah}
	\begin{align} \label{Inequalities-GM-Shah-RRN}
	\left| \frac{zp'(z)}{p(z)} \right| \leq \frac{2nr^n}{1-r^{2n}}
	\quad \text{ and } \quad
	\left| \frac{zq'(z)}{q(z)} \right| \leq \frac{2nr^n}{1-r^{2n}}.
	\end{align}
	Moreover, we have $f(z)=g(z)q(z)=zp(z)q(z)$, which gives the identity
	\begin{align} \label{Identity-Ratio-Fns-RRN}
	\mathcal{Q}_{f}(z)=\frac{zf'(z)}{f(z)}=1+\frac{zp'(z)}{p(z)}+\frac{zq'(z)}{q(z)}.
	\end{align}
	Using the inequalities \eqref{Inequalities-GM-Shah-RRN}, the above identity yields the disk
	\begin{align*}
	\left| \mathcal{Q}_{f}(z)-1 \right| \leq {4nr^n}/{(1-r^{2n})}.
	\end{align*}
	In view of \Cref{Lemma-Radius-ra-NeP}, the quantity $\mathcal{Q}_{f}(z)$ takes values from the region bounded by the nephroid \eqref{Equation-of-Nephroid} provided
	$4nr^n/(1-r^{2n})\leq2/3$, or equivalently, $r^{2n}+6nr^n-1\leq0$. This further gives $r\leq \rho_{\scriptscriptstyle{1}}$. To verify the sharpness of the bound $\rho_{\scriptscriptstyle{1}}$, consider the functions $f_1,g_1\in\mathcal{A}_n$ given by
	\begin{align*}
	f_1(z)=z\left(\frac{1+z^n}{1-z^n}\right)^2
	\quad \text{ and } \quad
	g_1(z)=z\left(\frac{1+z^n}{1-z^n}\right).
	\end{align*}
	It is clear that $f_1(z)/g_1(z)=g_1(z)/z=(1+z^n)/(1-z^n)\in\mathcal{P}_n$, and hence $f_1\in\mathcal{G}_1$. Also,
	\begin{align*}
	\mathcal{Q}_{f_{\scriptscriptstyle{1}}}(z)=1+{4nz^n}/{(1-z^{2n})},
	\end{align*}
	and at the points $z=\pm\rho_{\scriptscriptstyle{1}}\in\partial\mathbb{D}_{\rho_1}$, we have
	\begin{align*}
	\mathcal{Q}_{f_1}(-\rho_{\scriptscriptstyle{1}})=1/3\in\partial\Omega_{Ne}
	~\text{ and }~
	\mathcal{Q}_{f_1}(\rho_{\scriptscriptstyle{1}})=5/3\in\partial\Omega_{Ne}.
	\end{align*}
	 This proves sharpness of the estimate.
\item
	Let $f\in\mathcal{G}_2$. Define the functions $p,q:\mathbb{D}\to\mathbb{C}$ by $p=g/z$ and $q=f/g$. Then $f(z)=zp(z)q(z)$ with $p\in\mathcal{P}_n(1/2)$ and $q\in\mathcal{P}_n$. In light of the identity \eqref{Identity-Ratio-Fns-RRN}, it follows from \Cref{Lemmas-Combined-Ravi-1997-Shah-1972}\eqref{Lemmas-Combined-b-Shah}  that
	\begin{align*}
	\left|\mathcal{Q}_{f}(z)-1\right| \leq \frac{2nr^n}{1-r^{2n}}+\frac{nr^n}{1-r^{n}}=\frac{3nr^n+nr^{2n}}{1-r^{2n}}.
	\end{align*}
	Now, from \Cref{Lemma-Radius-ra-NeP}, $f\in\mathcal{S}^*_{Ne}$ if
	\begin{align*}
	\frac{3nr^n+nr^n}{1-r^{2n}}\leq\frac{2}{3}.
	\end{align*}
	That is, if $(3n+2)r^{2n}+9nr^n-2\leq0$, which gives the desired result $r\leq\rho_{\scriptscriptstyle{2}}$. To prove that $\rho_{\scriptscriptstyle{2}}$ is the sharp $\mathcal{S}^*_{Ne,n}$-radius for the function family $\mathcal{G}_2$, define
	\begin{align*}
	f_2(z)=\frac{z(1+z^n)}{(1-z^n)^2}
	\quad \text{ and } \quad
	g_2(z)=\frac{z}{1-z^n}.
	\end{align*}
	It is easy to see that ${f_2}/{g_2}\in\mathcal{P}_n$ and ${g_2}/{z}\in\mathcal{P}_n(1/2)$, which shows that $f_2$ is a member of $\mathcal{G}_2$. Further,
	\begin{align*}
	\mathcal{Q}_{f_2}(z)\big{|}_{z=\rho_2}=\frac{1+3nz^n+(n-1)z^{2n}}{1-z^{2n}}\Big{|}_{z=\rho_2}=\frac{5}{3}.
	\end{align*}
	This shows that
	\begin{align*}
	\mathcal{Q}_{f_2}\left(\partial\mathbb{D}_{\rho_2}\right)\cap\partial\Omega_{Ne}=\{5/3\}\neq\emptyset,
	\end{align*}	
	and hence the estimate is sharp.
\item
	Let $f\in\mathcal{G}_3$. Further, suppose that $p(z)={g(z)}/{z}$ and $q(z)={g(z)}/{f(z)}$. Obviously $p\in\mathcal{P}_n$, and $q\in\mathcal{P}_n(1/2)$ due to the fact that
	\begin{align*}
	\left|\frac{f(z)}{g(z)}-1\right|<1  \iff \frac{g}{f}\in\mathcal{P}_n\left(\frac{1}{2}\right).
	\end{align*}
	Verify that $f(z)=zp(z)/q(z)$ and
	\begin{align*}
	\mathcal{Q}_{f}(z)=1+\frac{zp'(z)}{p(z)}-\frac{zq'(z)}{q(z)}.
	\end{align*}
	Applying \Cref{Lemmas-Combined-Ravi-1997-Shah-1972}\eqref{Lemmas-Combined-b-Shah}  and simplifying, we obtain
	\begin{align*}
	\left|\mathcal{Q}_{f}(z)-1\right| \leq {(3nr^n+nr^{2n})}/{(1-r^{2n})} \leq {2}/{3}
	\end{align*}
	if $(3n+2)r^{2n}+9nr^n-2\leq0$, or, $r\leq\rho_{\scriptscriptstyle{3}}$. Therefore, by \Cref{Lemma-Radius-ra-NeP}, $f\in\mathcal{S}^*_{Ne}$ provided $r\leq\rho_{\scriptscriptstyle{3}}$. For sharpness, consider
	\begin{align*}
	f_3(z)=\frac{z(1+z^n)^2}{1-z^n}
	\quad \text{ and } \quad
	g_3(z)=\frac{z(1+z^n)}{1-z^n}
	\end{align*}
	satisfying
	\begin{align*}
	\left|\frac{f_3(z)}{g_3(z)}-1\right|=|z|^n<1
	\quad \text{ and } \quad
	\frac{g_3(z)}{z}=\frac{1+z^n}{1-z^n}\in\mathcal{P}_n,
	\end{align*}
	so that $f_3\in\mathcal{G}_3$. Moreover, it is easy to see that $\mathcal{Q}_{f_3}(z)$ assumes the value $1/3$ at the point  $z=\rho_3e^{i\pi/n}\in\partial\mathbb{D}_{\rho_3}$. This shows that the estimate is best possible.
\item
	Let $f\in\mathcal{G}_4$ and let $q:\mathbb{D}\to\mathbb{C}$ be given by $q=g/f$, where $g\in\mathcal{A}_n$ is some convex function. As earlier, we have $q\in\mathcal{P}_n(1/2)$ and hence \Cref{Lemmas-Combined-Ravi-1997-Shah-1972}\eqref{Lemmas-Combined-b-Shah}  gives
	\begin{align} \label{Inequalities2-GM-Shah-RRN}
	\left|\frac{zq'(z)}{q(z)}\right| \leq \frac{nr^n}{1-r^n}.
	\end{align}
	Further, since convexity of $g$ implies $zg'/g\in\mathcal{P}_n(1/2)$, we have the following inequality from \Cref{Lemmas-Combined-Ravi-1997-Shah-1972}\eqref{Lemmas-Combined-a-Ravi},	
	\begin{align} \label{Inequalities-Ronning-1997-RRN}
	\left|\frac{zg'(z)}{g(z)}-\frac{1}{1-r^{2n}}\right| \leq \frac{r^n}{1-r^{2n}}.
	\end{align}
	In view of the inequalities \eqref{Inequalities2-GM-Shah-RRN} and \eqref{Inequalities-Ronning-1997-RRN}, and the identity
	\begin{align*}
	\mathcal{Q}_f(z)=\frac{zg'(z)}{g(z)}-\frac{zq'(z)}{q(z)},
	\end{align*}
	we calculate that
	\begin{align} \label{Disk-Last-Part-Ratio-Fns}
	\left|\mathcal{Q}_f(z)-\frac{1}{1-r^{2n}}\right| \leq  \frac{(n+1)r^n+nr^{2n}}{1-r^{2n}}.
	\end{align}
	As the number $1-r^{2n}<1$ for each $r\in(0,1)$, it follows from \Cref{Lemma-Radius-ra-NeP} that the disc \eqref{Disk-Last-Part-Ratio-Fns} lies in the interior of the nephroid domain $\Omega_{Ne}$ if
	\begin{align*}
	\frac{(n+1)r^n+nr^{2n}}{1-r^{2n}} \leq \frac{1}{1-r^{2n}}-\frac{1}{3}.
	\end{align*}
	Or equivalently, if
	\begin{align*}
	(3n-1)r^{2n}+3(n+1)r^n-2\leq0.
	\end{align*}
	This further gives $r\leq\rho_{\scriptscriptstyle{4}}$. For sharpness of the estimate, consider
	\begin{align*}
	f_4(z)=\frac{z(1+z^n)}{(1-z^n)^{1/n}}
	\quad \text{ and } \quad
	g_4(z)=\frac{z}{(1-z^n)^{1/n}}.
	\end{align*}
	Since $\left|{f_4(z)}/{g_4(z)}-1\right|=|z|^n<1$ and $g_4\in\mathcal{C}_n$, we have $f_4\in\mathcal{G}_4$. Moreover, $\mathcal{Q}_{f_4}(-\rho_{\scriptscriptstyle{4}})=1/3$ and $\mathcal{Q}_{f_4}(\rho_{\scriptscriptstyle{4}})=5/3$, i,e.,
	 \begin{align*}
	 \mathcal{Q}_{f_4}\left(\partial\mathbb{D}_{\rho_4}\right)\cap\partial\Omega_{Ne}=\{1/3,5/3\}\neq\emptyset.
	 \end{align*}
	Therefore, the estimate is sharp.
\qedhere
\end{enumerate}
\end{proof}


In continuation of our discussion, we provide radii results for two classes of functions that possess nice geometrical properties. These classes had less attention with respect to properties other than the radii results, after the celebrated de Branges's theorem.
\subsection*{The close-to-starlike class $\mathcal{CS}^*_n(\alpha)$}
For $\alpha\in[0,1)$, the class of close-to-starlike functions $\mathcal{CS}^*_n(\alpha)$ of type $\alpha$ was introduced by Reade \cite{Reade-1955-Close-to-Convex} and is defined as
\begin{align*}
\mathcal{CS}^*(\alpha):=\left\{f\in\mathcal{A}_n:\frac{f}{g}\in\mathcal{P}_n,\, g\in\mathcal{S}^*_n(\alpha)\right\}.
\end{align*}
The following theorem determines the $\mathcal{S}^*_{Ne,n}$-radius for $\mathcal{CS}^*_n(\alpha)$.
\begin{theorem} \label{Theorem-Radius-Close-to-Stralike-RRN}
	The sharp $\mathcal{S}^*_{Ne,n}$-radius for $\mathcal{CS}^*_n(\alpha)$ is given by
	\begin{align*}
	 {R}_{\mathcal{S}^*_{Ne,n}}\left(\mathcal{CS}^*_n(\alpha)\right)=\rho_{cs}:=\left(\frac{2}{3(1+n-\alpha)+\sqrt{9(1+n-\alpha)^2+4(4-3\alpha)}}\right)^{1/n}.
	\end{align*}
\end{theorem}

\begin{proof}
	If $f\in\mathcal{CS}^*_n(\alpha)$, then $h={f}/{g}\in\mathcal{P}_n$ for some function $g\in\mathcal{S}^*_n(\alpha)$. Applying \Cref{Lemmas-Combined-Ravi-1997-Shah-1972}\eqref{Lemmas-Combined-b-Shah}, we have
	\begin{align} \label{Bound-Class-P-MacGregor-1963}
	\left|\frac{zh'(z)}{h(z)}\right| \leq\frac{2nr^n}{1-r^{2n}}.
	\end{align}
	Also, $g\in\mathcal{S}^*_n(\alpha)$ implies ${zg'}/{g}\in\mathcal{P}_n(\alpha)$, so that \Cref{Lemmas-Combined-Ravi-1997-Shah-1972}\eqref{Lemmas-Combined-a-Ravi} gives
	\begin{align} \label{Bound-Class-P-Alpha-Ravi-Ronning-1997}
	\left|\frac{zg'(z)}{g(z)}-\frac{1+(1-2{\alpha})r^{2n}}{1-r^{2n}}\right| \leq \frac{2(1-\alpha)r^n}{1-r^{2n}}.
	\end{align}
	In view of the identity
	\begin{align*}
	\mathcal{Q}_f(z)=\frac{zg'(z)}{g(z)}+\frac{zh'(z)}{h(z)},
	\end{align*}
	it follows from \eqref{Bound-Class-P-MacGregor-1963} and \eqref{Bound-Class-P-Alpha-Ravi-Ronning-1997} that
	\begin{align} \label{Disk-Radius-Close-to-Starlike}
	\left|\mathcal{Q}_f(z)-\frac{1+(1-2{\alpha})r^{2n}}{1-r^{2n}}\right| \leq \frac{2(1+n-\alpha)r^n}{1-r^{2n}}.
	\end{align}
	The inequality \eqref{Disk-Radius-Close-to-Starlike} represents a disk with center $a$ and radius $r_0$ given by
	\begin{align*}
	a:=\frac{1+(1-2{\alpha})r^{2n}}{1-r^{2n}} ~ \text{ and } ~ r_0:=\frac{2(1+n-\alpha)r^n}{1-r^{2n}}.
	\end{align*}	
	As $a>1$, we conclude from \Cref{Lemma-Radius-ra-NeP} that the disk \eqref{Disk-Radius-Close-to-Starlike} lies in the interior of the region $\Omega_{Ne}$ bounded by the nephroid \eqref{Equation-of-Nephroid} if $r_0\leq{5}/{3}-a$. That is, if
	 \begin{align*}
	 (4-3\alpha)r^{2n}+3(1+n-\alpha)r^n-1\leq0.
	 \end{align*}
	On solving this inequality we obtain $r\leq\rho_{cs}$. For sharpness, consider the functions $f,g\in\mathcal{A}_n$ given by
	\begin{align*}
	f(z)=\frac{z(1+z^n)}{(1-z^n)^{{(n+2-2\alpha)}/{n}}} \quad  \text{ and } \quad  g(z)=\frac{z}{(1-z^n)^{(2-2\alpha)/n}}.
	\end{align*}
	Since $f/g=(1+z^n)/(1-z^n)\in\mathcal{P}_n$ and $g\in\mathcal{S}^*_n(\alpha)$, the function $f\in\mathcal{CS}^*_n(\alpha)$. For this function $f$, we have $\mathcal{Q}_{f}(\rho_{cs})=5/3$. Thus the estimate is sharp.
\end{proof}

\subsection*{The MacGregor class $\mathcal{W}_n$}
Define the class $\mathcal{W}_n$ as
\begin{align*}
\mathcal{W}_n:=\left\{f\in\mathcal{A}_n:{f}/{z}\in\mathcal{P}_n\right\}.
\end{align*}
Introduced first by MacGregor \cite{MacGregor-1963-PAMS}, this function class was recently studied for several radius problems by the authors in \cite{Ali-Jain-Ravi-2012-Radii-LemB-AMC,Cho-2019-Sine-BIMS,Mendiratta-2014-Shifted-Lemn-Bernoulli,
	Mendiratta-Ravi-2015-Expo-BMMS}. Following them, we prove the following sharp radius result for $\mathcal{W}_n$.

\begin{theorem}
	The sharp $\mathcal{S}^*_{Ne,n}$-radius for $\mathcal{W}_n$ is
	\begin{align*}
	 {R}_{\mathcal{S}^*_{Ne,n}}\left(\mathcal{W}_n\right)=\rho_{\scriptscriptstyle{\mathcal{W}}}:=\left({2}/{(3n+\sqrt{9n^2+4})}\right)^{1/n}.
	\end{align*}
\end{theorem}
\begin{proof}
	Let $f\in\mathcal{W}_n$ and $|z|=r<1$. Then $h(z)={f(z)}/{z}\in\mathcal{P}_n$ and \Cref{Lemmas-Combined-Ravi-1997-Shah-1972}\eqref{Lemmas-Combined-b-Shah}  gives
	\begin{align*}
	\left|\frac{zh'(z)}{h(z)}\right| \leq\frac{2nr^n}{1-r^{2n}}.
	\end{align*}
	On using the above inequality in the identity
	$\mathcal{Q}_f(z)=1+{zh'(z)}/{h(z)}$,
	we obtain
	\begin{align*}
	\left|\mathcal{Q}_f(z)-1\right|\leq {2nr^n}/{(1-r^{2n})}.
	\end{align*}
	Now applying \Cref{Lemma-Radius-ra-NeP}, it follows that $\mathcal{Q}_f\in\Omega_{Ne}$
	if $2nr^n/(1-r^{2n})\leq2/3$, or, if $r^{2n}+3nr^n-1\leq0$. The last inequality yields $r\leq\rho_{\scriptscriptstyle{\mathcal{M}}}$. For the function $f_{\scriptscriptstyle{\mathcal{W}}}(z)=z(1+z^n)/(1-z^n)$ satisfying $f(z)/z\in\mathcal{P}_n$, we have $\mathcal{Q}_{f_\mathcal{W}}(-\rho_{\scriptscriptstyle{\mathcal{W}}})=1/3$ and $\mathcal{Q}_{f_\mathcal{W}}(\rho_{\scriptscriptstyle{\mathcal{W}}})=5/3$. Therefore,
	\begin{align*}
	 \mathcal{Q}_{f_\mathcal{W}}\left(\partial\mathbb{D}_{\rho_\mathcal{W}}\right)\cap\partial\Omega_{Ne}=\{1/3,5/3\}\neq\emptyset,
	\end{align*}
	and hence the estimate is best possible.	
\end{proof}

Besides the above mentioned two classes, similar type of radius results for various other classes can be obtained. Among these, as an exhibit, the authors favor the following class.
\subsection*{The class $\mathcal{M}_n(\beta)$}
For $\beta>1$, the family of functions $\mathcal{M}_n(\beta)$ introduced by Uralegaddi et al. \cite{Uralegaddi-Ganigi-Sarangi} is defined as
\begin{align*}
\mathcal{M}_n(\beta):=\left\{ f\in\mathcal{A}_n:\mathrm{Re}\left(\mathcal{Q}_f(z)\right)<\beta,\; z\in\mathbb{D} \right\}.
\end{align*}
In terms of subordination, $\mathcal{M}_n(\beta)$ has the following characterization:
\begin{align*}
f\in\mathcal{M}_n(\beta) \iff  \mathcal{Q}_f(z)\prec \frac{1+(1-2\beta)z^n}{1-z^n}
,  \qquad \beta>1.
\end{align*}
The class $\mathcal{M}_n(\beta)$ gained less attention from the researchers working in this direction. Nevertheless, it compeers the same significance, in analytic properties, as that of the class of starlike functions with negative coefficients introduced by Silverman \cite{Silverman-1975-UFs-Neg-Coeficients-PAMS}. For this interesting class of functions, we have the following result without details of the proof, as the proof can be traced in lines similar to the proof of \Cref{Theorem-Radius-Close-to-Stralike-RRN}.

\begin{theorem}
	The $\mathcal{S}^*_{Ne,n}$-radius for the class $\mathcal{M}_n(\beta)$ is given by
	\begin{align*}
	{R}_{\mathcal{S}^*_{Ne,n}}\left(\mathcal{M}_n(\beta)\right)=
	(3\beta-2)^{-1/n}.
	\end{align*}
	The result is sharp for the function $f(z)=z(1-z^n)^{2(\beta-1)/n}$.
\end{theorem}



\end{document}